\documentclass{amsart}

\usepackage{amsmath}
\usepackage{amsfonts}
\usepackage{amsthm}
\usepackage{commath}
\usepackage{graphicx}
\usepackage[colorlinks=true, allcolors=blue]{hyperref}
\usepackage[english]{babel}

\usepackage[letterpaper,top=2cm,bottom=2cm,left=3cm,right=3cm,marginparwidth=1.75cm]{geometry}

\renewcommand{\a}[1]{\abs{#1}}
\newcommand{\f}[2]{\frac{#1}{#2}}
\newcommand{\R}[0]{\mathbb{R}}
\newcommand{\C}[0]{\mathbb{C}}
\newcommand{\mf}[1]{\mathfrak{#1}}
\newcommand{\mc}[1]{\mathcal{#1}}
\newcommand{\1}{\mathbf{1}}

\DeclareMathOperator*{\bigast}{\text{\scalebox{1.5}{$\ast$}}}
\DeclareMathOperator\supp{supp}

\renewcommand{\Re}{\operatorname{Re}}
\DeclareMathOperator{\Tr}{\operatorname{Tr}}
\renewcommand{\l}{\lambda}

\newcommand{\Sym}[0]{\text{Sym}}

\newtheorem{lem}{Lemma}

\newtheorem{thm}{Theorem}
\newcommand{\xt}[2]{\kappa_{#1,#2}}
\renewcommand{\H}{\mc{H}}
\newcommand{\answer}{A}

\begin{document}

\title{Uniform Bounds for Maximal Flat Periods on $SL(n,\R)$}
\author{Phillip Harris}

\begin{abstract}
Let $X$ be a compact locally symmetric space associated to $SL(n,\R)$ and $Y \subset X$ a maximal flat submanifold, not necessarily closed. Using a Euclidean approximation, we give an upper bound in the spectral aspect for Maass forms integrated against a smooth cutoff function on $Y$.
\end{abstract}

\maketitle

\section{Introduction}
\pagenumbering{arabic} 

Let $G$ be a noncompact semisimple Lie group with Iwasawa decomposition $G = NAK$ and Lie algebras $\mf{g} = \mf{n + a + k}$.  Let $G/K$ be the associated symmetric space and $X = \Gamma \backslash G / K$ a compact quotient. 
Let $(f_i)_i \in L^2(X)$ be an orthonormal basis of Maass forms with spectral parameters $\nu_i \in \mf{a}^*_\C$. 
Let $Y \subset X$ be a maximal flat submanifold (not necessarily closed) and $b \in C_c^\infty(Y)$ a smooth cutoff function. We are interested in the growth of the \textit{flat periods} 
\renewcommand{\P}{\mc{P}}
\[
\P_i = \int_Y f_i(x) b(x) dx
\]
with respect to $\nu_i$. 

There is a general bound for Laplace eigenfunctions on compact manifolds due to Zelditch \cite{zelditch} which in our case gives
\[
\abs{\P_i} \ll \norm{\nu_i}^{(\dim X - \dim Y - 1)/2} = \norm{\nu_i}^{(n^2-n-2)/4}
.\]

 Michels \cite{bart} was the first to study flat periods on higher rank locally symmetric spaces. His work implies\footnote{This result for Maass forms follows from applying the trace formula argument in Section \ref{sec:traceformula} to Michels' bound for spherical functions. } an averaged estimate for flat periods, as follows.
 Let $X$ be a compact locally symmetric space of noncompact type, of rank $r$. 
 Let $\Sigma \subset \mf{a}^*$ be the restricted roots of $X$ and $\Sigma^+$ be the positive roots.  
 Define the set of ``generic points'' $(\mf{a^*})^\text{gen} \subset \mf{a^*}$ to be the set of points that are regular and that do not lie in any proper subspace spanned by roots. 
  Fix a closed cone $D \subset (\mf{a^*})^\text{gen}$. Let $\beta : \mf{a}^* \to \R$ be the Plancherel density. Then there exists $C>0$ such that uniformly for $\l \in D$ we have
\[
\sum_{\norm{\Re \nu_i - \lambda} \leq C} \abs{\mc{P}_i}^2 \ll \beta(\lambda) (1+\norm{\l})^{-r}
.\]
In other words, consider all periods $\mc{P}_i$ such that $\Re \nu_i$ lies in a ball of fixed radius around $\l$, then their average norm squared is $\ll  (1+\norm{\l})^{-r}$. For certain choices of $X$ and $Y$ enumerated in \cite[Theorem 1.3]{bart} Michels proves a lower bound (replacing $\ll$ with $\asymp$).

For $X$ associated to $SL(n,\R)$, we give an estimate uniform in $\l$, elucidating the behavior on the non-generic set. 

\begin{thm}\label{thm:main}
Uniformly for $\l = \textup{diag}(\l_1,\dots,\l_n) \in \mf{a}^*$ we have
\begin{align*}\label{eq:main}
\sum_{\norm{\Re \nu_i - \l} \leq 1} \abs{\P_i}^2 \ll \tilde{\beta}(\l) (1+\norm{\l})^{-n+1}
L_n(\l)
\end{align*}
with the implied constant depending only on $n$ and the test function $b$. Here $\tilde{\beta}(\lambda) = \prod_{\alpha \in \Sigma^+} 1+\abs{\langle \lambda , \alpha \rangle}$ is essentially the maximum of $\beta$ taken over a ball around $\l$. The function $L_n(\l)$ is Weyl invariant, and in the Weyl chamber \footnote{This is opposite from the canonical positive chamber.} with $\l_1 \leq \dots \leq \l_n$, we define it as follows. Define $\log' x = \log(2 + x)$. For $n \neq 4$,
\[
L_n(\l) := \left( \log' \frac{\norm{\l}}{1 + \abs{\l_2} + \abs{\l_{n-1}}} \right)^{n-2}
,
\]
and for $n = 4$,
\[
L_4(\l) := 
\left( \log' \frac{\norm{\l}}{1 + \abs{\l_2} + \abs{\l_{3}}} \right)^2
\log' \frac{\norm{\l}}{1 + \abs{\l_1 - \l_2} + \abs{\l_3 - \l_4}}
.
\]
\end{thm}

Since $\tilde{\beta}(\l) \ll \norm{\l}^{n(n-1)/2}$ we obtain $\abs{\P_i} \ll \norm{\nu_i}^{(n^2-3n+2)/4}L_n(\nu_i)^{1/2}$, an improvement on the bound for general manifolds. The factor $L_n(\l)$ can be given a geometric interpretation, namely, it grows when $\l$ is collinear with a root. In the $n=4$ case, it grows when $\l$ is collinear with a root or a sum of two orthogonal roots.

\subsection{Outline of the proof}

Using a standard pre-trace formula argument, we reduce the period estimate to an estimate for spherical functions: let $A \subset SL(n,\R)$ be the diagonal subgroup and $b \in C^\infty_c(A)$, then
\begin{equation}\label{eqn:spherical}
\int_A \varphi_\lambda(g) b(g) dg \ll_b 
(1+\norm{\lambda})^{-n+1} 
L_n(\l)
.
\end{equation}
We use a theorem of Duistermaat to linearize the problem, replacing the spherical function $\varphi_\lambda$ on the Lie group with a ``Euclidean'' approximation on the Lie algebra. 
The problem then reduces to the following random matrix lemma:
Let $\l \in \mf{a}^*$ be regular. For a real $n \times n$ matrix $A$, let $d(A) = \sum_{1 \leq i \leq n} A_{ii}^2$ be the size of the diagonal. Choose $k \in SO(n)$ at random according to the Haar measure, then 
\[
\textup{Prob}[ d(k \l k^{-1}) < 1] \asymp (1+\norm{\l})^{-n+1} L_n(\l).
\]
While we were unable to prove a lower bound in Theorem \ref{thm:main}, this lemma and Michels' lower bounds are at least suggestive that our upper bound is sharp. 

\section{Proof of theorem 1}

\subsection{Preliminaries and Notation}
The notation $A \ll B$ means there exists $C > 0$ such that $A \leq CB$, and $A \asymp B$ means $A \ll B \ll A$. The implied constant $C$ may always depend on the dimension $n$ and bump function $b$; any other dependency is indicated with a subscript. When $A$ is a matrix, $\norm{A}$ will denote the Frobenius norm $\norm{A}^2 = \sum_{ij} \abs{A_{ij}}^2$. Indicator functions are denoted $\1_S(x)$ where $S$ is a set or $\1[P(x)]$ where $P$ is a predicate. 

Let $G = SL(n,\R)$. We write the Iwasawa decomposition $G = NAK$, where $N$ is the upper triangular subgroup with 1's on the diagonal, $A$ is the diagonal subgroup and $K$ is $SO(n)$. We have the corresponding Lie algebra decomposition $\mf{g = n + a + k}$ and the Cartan decomposition $\mf{g = p + k}$. Then $X = \Gamma \backslash G / K$ for some cocompact lattice $\Gamma \subset G$, and $Y$ lies in the image of $gA$ for some $g \in G$. 
Using a partition of unity we may assume the support of $b$ is small, say, having diameter $< R/100$ where $R$ is the injectivity radius of $X$, and lift $b$ to $C^\infty_c(gA)$.

\subsection{Spherical Functions and the Trace Formula}\label{sec:traceformula}

Let $H : G \to \mf{a}$ be the smooth map satisfying $g \in N e^{H(g)} K$ for all $g \in G$ (sometimes called the \textit{Iwasawa projection} \cite{bart}). Let $\rho$ be the half sum of the positive roots and $W$ be the Weyl group. Recall the spherical functions
\[
\varphi_\lambda(g) = \int_K e^{(i\lambda + \rho)H(kg)} dk
,
\]
the Harish-Chandra transform 
\[
\widehat{f}(\lambda) = \int_G f(x) \varphi_{-\lambda}(x) dx \,\,\,\,\,\,\,\,\, f \in C^\infty_c(K \backslash G/K)
,
\]
and its inverse
\[
f(x) = 
\frac{1}{\a{W}}
\int_\mf{a^*} \varphi_\lambda(x) \widehat{f}(\lambda) \beta(\lambda) d\lambda
.\]
In order to prove Theorem \ref{thm:main} we use the pre-trace formula for $SL(n,\R)$. Let $k \in C^\infty_c(K\backslash G / K)$ be a bi-$K$-invariant test function and  
\[
K(x,y) = \sum_{\gamma \in \Gamma} k(x^{-1} \gamma y)
\]
be the corresponding automorphic kernel on $X$. Then the pre-trace formula \cite{pretrace} states 
\[
\sum_{\gamma \in \Gamma} k(x^{-1} \gamma y) = \sum_i \widehat{k}(- \nu_i) f_i(x) \overline{f_i(y)}
\]
where $\widehat{k}$ is the Harish-Chandra transform of $k$. 
Integrating against $b(x)b(y) \in C_c^\infty(gA \times gA)$ we obtain
\begin{equation}\label{eq:pretrace}
  \int_{gA} \int_{gA} \sum_{\gamma \in \Gamma} 
  k(x^{-1} \gamma y)
   b(x) b(y)
  dx dy = \sum_i \widehat{k}(- \nu_i) \abs{\P_i}^2  
.\end{equation}
 
 Next we construct a $k$ such that $\widehat{k}$ concentrates around a given $\l_0$. Recall that the spectral parameters of Maass forms satisfy $W \nu_i = W \overline{\nu_i}$ and $\norm{\operatorname{Im} \nu_i} \leq \norm{\rho}$. Define $\Omega = \{ \l \in \mf{a}^*_\C : W \l = W \overline{\l}, \norm{\operatorname{Im} \l} \leq \norm{\rho} \}$ \cite[Lemma 4.1]{marshallPaley}. Using the Paley-Wiener theorem for the Harish-Chandra transform we construct $k$ with the following properties:
\begin{enumerate}
    \item $k$ is supported in a ball of radius $R/100$, where $R$ is the injectivity radius of $X$. 
    \item $\widehat{k}(\l) \geq 0$ for $\l \in \Omega$.
    \item $\widehat{k}(\l) \geq 1$ for $\l \in \Omega$ with $\norm{\text{Re } \l - \l_0} \leq 1$. 
    \item $\widehat{k}(\l) \ll_N (1+\norm{\l-
    \l_0})^{-N}$ for $\nu \in \Omega$ uniformly in $\l_0$. 
\end{enumerate}
The details of this construction may be found in \cite[Section 4.1]{marshallPaley}. 

By property (1),  all terms except $\gamma = e$ on the geometric side drop out. By properties (2) and (3) we have 
\[
\int_{gA} \int_{gA} k(x^{-1} y) b(x)b(y) dx dy
\geq 
\sum_{\norm{\Re \nu_i - \l_0} \leq 1} \abs{\P_i}^2
.\]
Note that $x^{-1} y \in A$. 
Defining $b_1(z) = \int_{gA} b(x) b(xz) dx \in C^\infty_c(A)$ and changing variables $z = x^{-1} y$ the left hand side becomes
\begin{align*}
\int_{gA} \int_{gA} k(x^{-1} y) b(x)b(y) dx dy
&=
\int_A k(z) b_1(z) \, dz 
\\
&= 
\int_A 
\left(
\f{1}{\a{W}}
\int_{\mf{a^*}}
\varphi_{\l}(z) \widehat{k}(\l) \beta(\l) 
\, d\mu
\right)
b_1(z)
dz 
\\
&=
\f{1}{\a{W}}
\int_{\mf{a^*}}
\widehat{k}(\l) \beta(\l)
\left(
\int_A 
\varphi_{\l}(z)
b_1(z) dz 
\right)
\, d\l
.
\end{align*}
The task is now to bound the integral of a spherical function $\varphi_\lambda$ against a smooth cutoff function on $A$. 
In other words, we need the following bound: for $b \in C^\infty_c(A)$, 
\[
\int_A \varphi_\l(z) b(z) dz \ll_b 
(1+\norm{\l})^{-n+1} 
L_n(\l)
.
\]
With this bound, Theorem \ref{thm:main} follows from the rapid decay of $\widehat{k}$ away from $-\l_0$ (property 4) and the polynomial growth of $\beta$. 

\subsection{Euclidean Approximation of Spherical Functions}
Let $\mf{g = k + p}$ be the
Cartan decomposition. For $X \in \mf{p}, k \in K$ write $k.X = k X k^{-1}$ for the adjoint action of $K$ on $\mf{p}$. 
Let $\pi : \mf{p} \to \mf{a}$ be the orthogonal projection with respect to the Killing form. By a theorem of Duistermaat \cite[Equation 1.11]{duistermaat}, there exists a nonnegative analytic function $a \in C^\infty(\mf{p})$ such that 
\[
\varphi_\lambda(e^X) = \int_K e^{
\langle i \lambda, \pi(k.X) \rangle} a(k.X) dk 
\]
where $dk$ is the Haar measure.
Changing variables $z = e^X$ and using Duistermaat's formula gives
\begin{align*}
(*) := 
\int_A \varphi_\lambda(z) b(z) dz
&=
\int_\mf{a} \int_K e^{\langle i\lambda, \pi(k.X)\rangle} 
a(k.X) b(e^X) 
dX 
dk 
\\
&=
\int_K
\left(
\int_\mf{a}
e^{\langle ik^{-1}.\lambda, X \rangle} 
a(k.X) b(e^X) 
dX 
\right)
dk 
,
\end{align*}
where in the second line we extend $\l \in \mf{a}^*$ to $\mf{p}^*$ by pulling back along $\pi$, and replace the adjoint action with the coadjoint action. 
Let $c(k, X) := a(k.X) b(e^X)$ and note that $c \in C^\infty_c(K \times \mf{a})$. Let $\widehat{c}(k, \xi) \in C^\infty(K \times \mf{a}^*)$ be the Fourier transform of $c$ in the second variable. Then the inner integral is just $\widehat{c}$ evaluated at the frequency $\xi = -\pi^*(k^{-1}.\l)$, where $\pi^* : \mf{p}^* \to \mf{a}^*$ is restriction: 
\[
(*) =
\int_K \widehat{c}(k, - \pi^*( k^{-1}.\lambda) ) dk
.\]
Since $\widehat{c}(k, \xi)$ has rapid decay in $\xi$ for all $k$ and $K$ is compact, $\sup_k \widehat{c}(k, \xi)$ also has rapid decay in $\xi$. Let $B$ be the indicator function of the unit ball on $\mf{a^*}$, pulled back along $\pi^*$ to $\mf{p^*}$. Taking the supremum in the first variable and upper bounding in the second variable by balls gives
\[
\abs{(*)} \leq \sum_{r=1}^\infty d_r \int_K B(r^{-1}(k^{-1}. \lambda)) dk 
\]
where the sequence $d_r$ has rapid decay. Essentially, we want to know how much the coadjoint orbit $K.\lambda$ can concentrate near $(\mf{a^*})^\perp$. 
It suffices to show the following lemma: 

\begin{lem}\label{lem:main}
    Let $\Sym_n(\R)$ be the set of real symmetric $n \times n$ matrices. For $X \in \Sym_n(\R)$ let $\pi(X)$ be the diagonal part of $X$ and $B(X) = \1[{ \norm{\pi(X)}<1}]$. Let $\l = \text{diag}(\l_1,\dots,\l_n)$ with $\l_1 < \l_2 < \dots < \l_n$. Define 
    \[
    I_n(\l) 
    :=
    \int_{SO(n)}
    B(k.\l)
    \, dk 
    \]
    and 
    \[
    \answer_n(\l) := (1+\norm{\l})^{-n+1} L_n(\l)
    .
    \]
    Then if $\Tr \l = 0$ we have $I_n(\l) \asymp A_n(\l)$.
\end{lem}
Finally, since $(*)$ is continuous in $\l$, we may perturb $\l$ to be regular (i.e. $\l_i < \l_{i+1}$ rather than $\leq)$, then 
\begin{align*}
\abs{(*)} 
&\leq \sum_{r=1}^\infty d_r (1+r^{-1}\norm{\lambda})^{-n+1} L_n(r^{-1}\l) 
\\
&\ll (1+\norm{\lambda})^{-n+1} L_n(\l)
.
\end{align*}
\section{The coadjoint orbit $K.\l$}

The remainder of the paper will be proving Lemma \ref{lem:main}.

\subsection{Preliminaries}

We note some properties of $I_n(\l)$ and $\answer_n(\l)$.
\begin{itemize}
    \item 
    $I_n(\l)$ is monotone under scaling: if $t \geq 1$ then $I_n(t\l) \leq I_n(\l)$.
    \item 
    For a fixed $C \geq 0$, we have $\answer_n(C \l) \asymp_C \answer_n(\l)$, and $\norm{\l - \l'} \leq C$ implies $\answer_n(\l) \asymp_C \answer_n(\l')$. 
    \item 
    An estimate for $I_n(\l)$ when $\Tr \l \neq 0$ easily follows from the tracefree case. Write $\l = \l_0 + (\Tr \l / n)I$. Then 
\begin{align*}
    \norm{\pi(k.\l)}^2 
    &= 
    \norm{\pi(k.\l_0)}^2
    + 
    2\langle 
    \pi(k.\l_0)
    ,
    (\Tr \l / n)I
    \rangle
    +
    \norm{\pi((\Tr \l / n)I)}^2
    \\
    &=
    \norm{\pi(k.\l_0)}^2
    +
    (\Tr \l)^2/n
    ,
\end{align*}
so $\norm{\pi(k.\l)} < 1$ if and only if $\norm{\pi(k.\l_0)} < \sqrt{1 - (\Tr \l)^2/n}$. Thus 
\begin{align}\label{eq:nonzerotrace}
I_n(\l) =
\begin{cases}
    I_n(\l_0 / \sqrt{1-(\Tr \l)^2/n}) & \a{\Tr \l} \leq \sqrt{n}
    \\
    0 & \a{\Tr \l} > \sqrt{n}
\end{cases}
.
\end{align}
\end{itemize}


We also note the following soft lower bound for $I_n(\l)$. 
\begin{lem}
     $I_n(\l) \gg (1+\norm{\l})^{-\dim SO(n)}$.
\end{lem}
\begin{proof}
An easy inductive argument shows that for all $\l$ with $\Tr \l = 0$, there exists some $k_0 \in SO(n)$ such that $\pi(k_0.\l) = 0$. Indeed, suppose $\l_1 < 0 < \l_2$, and let $R_\theta \in SO(2)$ act on the first two coordinates. Then acting by a quarter-turn swaps $\l_1$ and $\l_2$: 
\begin{align*}
    R_{\pi/2}. \begin{bmatrix}
        \l_1 &  & \\
         & \l_2 & \\
         & & \ddots 
    \end{bmatrix}
    &= 
\begin{bmatrix}
        \l_2 &  & \\
         & \l_1 & \\
         & & \ddots 
    \end{bmatrix},
\end{align*}
so by continuity there exists $\theta$ such that $R_\theta.\l$ has a 0 in the upper left corner, etc. 

Choose $X \in \mf{so}(n)$ with $\norm{X} < 1/100$. Taylor expansion gives 
\begin{align*}
    (\exp X)k_0.\l 
    &= 
    k_0.\l + X (k_0.\l) - (k_0.\l)X + \frac{1}{2} \left( 
    X^2(k_0.\l) -
    2X(k_0.\l)X +
    (k_0.\l)X^2
    \right) + \dots 
\end{align*}
Thus $\norm{\pi((\exp X) k_0 . \l)} \ll \norm{X}\norm{\l}$, and $\norm{\pi(k.\l)} < 1$ on a ball of radius $\gg \min(1/100,\norm{\l}^{-1})$ around $k_0$.
\end{proof}

Since $ (1+\norm{\l})^{-\dim SO(n)} \ll I_n(\l) \leq 1$, in the rest of the proof we may assume that $\norm{\l}$ is greater than some constant depending only on $n$ and prove $I_n(\l) \asymp \norm{\l}^{-n+1} L_n(\l)$ instead of $I_n(\l) \asymp (1+\norm{\l})^{-n+1} L_n(\l)$.



\subsection{Inductive step}

For $n=2$, we write $\l = \begin{bmatrix}
    -a & 0 \\ 0 & a 
\end{bmatrix}, k = \begin{bmatrix}
    \cos \theta & \sin \theta \\ -\sin \theta & \cos \theta 
\end{bmatrix}$ and easily evaluate the integral:
\begin{align*}
    \int_{SO(2)} B(k.\l) dk 
    &= \frac{1}{2\pi} \int_0^{2\pi} 
    B 
    \left( 
    \left[
    \begin{matrix}
        -a\cos 2\theta & a\sin 2 \theta
        \\ 
        a\sin 2 \theta & a\cos 2\theta 
    \end{matrix}
    \right]
    \right) 
    d\theta 
    \\
    &= 
    \frac{1}{2\pi}
    \int_0^{2\pi} 
    \bold{1}[2 a^2 \cos^2 (2\theta) < 1] d\theta 
    \\
    &\asymp (1 + \abs{a})^{-1}
    .
\end{align*}
So in the remainder of the proof assume $n \geq 3$.

Let $e_1, \dots, e_n$ be basis vectors for $\R^n$ and let $SO(n-1) \subset SO(n)$ be the subgroup fixing $e_n$. Let $\Pi : \Sym_n(\R) \to \Sym_{n-1}(\R)$ be the projection to the upper left $n-1\times n-1$ submatrix; we have $\Pi(k'.X) = k'.\Pi(X)$ for $k' \in SO(n-1)$. Let $\pi'$ and $B'$ be the $(n-1)$-dimensional versions of $\pi$ and $B$. Let $X = k.\l$. We have $\norm{\pi(X)}^2 = \norm{\pi'(\Pi(X))}^2 + X_{nn}^2$, so $B(X) \leq B'(\Pi(X))$ and 
\begin{align}\label{eq:indupper}
    I_n(\l) \leq \int_{SO(n)} B'(\Pi(k.\l)) \, dk. 
\end{align}
On the other hand, since $\Tr X = 0$ we have $X_{nn} = - \sum_{i=1}^{n-1} X_{ii}$ and by Cauchy-Schwarz $X_{nn}^2 \leq (n-1) \sum_{i=1}^{n-1} X_{ii}^2 = (n-1) \norm{\pi'(\Pi(X))}^2$, so $\norm{\pi(X)} \leq \sqrt{n} \norm{\pi'(\Pi(X))}$ and 
\begin{align*}
    I_n(\l) \geq \int_{SO(n)} B'(\sqrt{n}\Pi( k.\l)) \, dk. 
\end{align*}
To do the induction we will convert the integral over $SO(n)$ to an integral over $SO(n-1)$: 
\begin{lem}\label{lem:transfer}
    Define 
    \begin{align*}
        \mc{M} := \left\{ 
        \left[
        \begin{matrix}
            \mu_1 & & \\
             & \ddots & \\
             & & \mu_{n-1} 
        \end{matrix}
        \right]
        : \l_i < \mu_i < \l_{i+1} \right\}
    \end{align*}
    and 
    \begin{align*}
        J(\mu) := \frac{
        \prod_{1 \leq i < j \leq n-1}
        \a{\mu_i-\mu_j}
        }{
        \prod_{\substack{1 \leq i \leq n \\ 1 \leq j \leq n-1}} \a{\l_i - \mu_j}^{1/2}
        } 
        . 
    \end{align*}
    For any non-negative $f : \Sym_{n-1}(\R) \to \R$, we have
    \begin{align}\label{eq:transfer}
        \int_{SO(n)} f(\Pi(k.\l)) \, dk 
        = 
        c_n
        \int_\mc{M} 
        \int_{SO(n-1)} f(k'.\mu) \, dk' 
        J(\mu) \, d\mu 
        . 
    \end{align}
    for some absolute constant $c_n> 0$. 
\end{lem}
\begin{proof}
    Let $\mc{A} \subset \Sym_{n-1}(\R) $ be the set of matrices conjugate to some $\mu \in \mc{M}$, 
    equipped with the measure inherited from $\R^{(n-1)^2}$. We will show both sides are proportional to 
    \begin{align}\label{eq:middle}
        \int_{\mc{A}} f(X) 
        \prod_{\substack{1 \leq i \leq n \\ 1 \leq j \leq n - 1}} \a{\l_i - \mu_j(X)}^{-1/2}
        \, dX
    \end{align}
    where $\mu_j(X)$ is the $j$-th highest eigenvalue of $X$, viewed as a function on $\mc{A}$. Define a map 
    \begin{align*}
        F &: SO(n) \to \Sym_{n-1}(\R) 
        \\ 
        F(k) &= \Pi(k.\l) 
        .
    \end{align*}
    We begin by pushing the left hand side of (\ref{eq:transfer}) forward along $F$ to get (\ref{eq:middle}). First we exclude the degenerate case where $\Pi(k.\l)$ shares an eigenvalue with $\l$. 
    \begin{lem}\label{lem:quadrants}
    $\l_i$ is an eigenvalue of $\Pi(k.\l)$ if and only if $\langle e_i, k^{-1} e_n \rangle = 0$. 
\end{lem}
\begin{proof}
    By replacing $\l$ with $\l - \l_i I$, we may assume without loss of generality that $\l_i = 0$. Let $P = I - e_n e_n^T$, the projection killing the $n$-th coordinate. Then for any $A \in \Sym_n(\R)$,
    \begin{align*}
    P A P = 
    \begin{bmatrix}
    \Pi(A) & 0 
    \\ 
    0 & 0 
    \end{bmatrix}, 
    \end{align*}
    so the eigenvalues of $PAP$ are the eigenvalues of $\Pi(A)$, with an extra multiplicity at $0$. 
    
    Suppose $\l_i$ is an eigenvalue of $\Pi(k.\l)$, then $Pk\l k^{-1}Pv = 0$ for some $v$ with $Pv \neq 0$. If $\l k^{-1} P v = 0$, then $k^{-1} P v = c e_i$ for some $c \neq 0$, and  $\langle  e_i, k^{-1} e_n \rangle = c^{-1} \langle k^{-1} P v , k^{-1} e_n \rangle = c^{-1} \langle Pv, e_n \rangle = c^{-1} \langle v, P e_n \rangle = 0$. Otherwise, if $P k \l k^{-1} P v = 0$ and $\l k^{-1} P v \neq 0$ then $k \l k^{-1} P v = c e_n$ for some $c \neq 0$ and $\langle e_i, k^{-1} e_n \rangle = c^{-1} \langle e_i, \l k^{-1} P v \rangle = c^{-1} \langle \l e_i, k^{-1} P v \rangle = 0$. 
    
    Conversely, if $\langle e_i, k^{-1} e_n \rangle = 0$ then $P k e_i = k e_i$ and $Pk\l k^{-1} P k e_i = Pk\l  e_i = 0$, so $k e_i$ is the desired eigenvector for $\Pi(k.\l)$. 
    \end{proof}
    
    Let $\mc{S} = \{ k \in SO(n) : \langle e_i , k^{-1} e_n \rangle = 0 \text{ for some $i$ } \}$. Note that the quotient $SO(n-1) \backslash \mc{S}$ is the intersection of the sphere $S^n \simeq SO(n-1) \backslash SO(n)$ with the coordinate planes, so $\mc{S}$ is negligible. 
    A calculation implicit in Fan and Pall \cite[Theorem 1 and pp.300-301]{fanpall} shows that $F(SO(n) - \mc{S}) = \mc{A}$. 
    \begin{lem}[Fan-Pall]
        Let $\mu_1 \leq \dots \leq \mu_{n-1}$ and $z_1, \dots, z_n \in \R$. Suppose the matrix 
        \begin{equation}\label{eq:fanpallmatrix}
    \left[
        \begin{matrix}
            \mu_1 &        &             & z_1 \\ 
                    & \ddots &             & \vdots \\ 
                    &        & \mu_{n-1} & z_{n-1} \\ 
              z_1   &  \dots & z_{n-1} & z_n \\ 
        \end{matrix}
        \right]
    \end{equation}
    has eigenvalues $\l_1 \leq \dots \leq \l_n$. Then $\l_1 \leq \mu_1 \leq \l_2 \leq \dots \leq \mu_{n-1} \leq \l_n$.
    Furthermore, if $\l_i \neq \mu_j$ for all $i,j$ then
    \[
    z_i^2 = \frac{
    \prod_{1 \leq j \leq n}
    \abs{\l_j - \mu_i}
    }{
    \prod_{\substack{
    1 \leq j \leq n-1
    \\
    i \neq j
    }}
    \abs{\mu_i - \mu_j}
    } 
    \]
    for $1 \leq i \leq n-1$, and taking traces gives $z_n = \sum_i \l_i - \sum_j \mu_j$. 
    
    Conversely, for any $\mu$ and $\l$ satisfying $\l_i \neq \mu_j$ the above choice of $z_i$ gives a matrix with eigenvalues $\l$. 
    \end{lem}
    
    We check that $F$ restricted to $SO(n) - \mc{S}$ is a smooth covering and compute its differential. Let $Y \in \mc{A}$. Conjugating by $SO(n-1)$ we may assume $Y$ is diagonal. There are $2^{n-1}$ choices of $X \in \Sym_n(\R)$ conjugate to $\l$ such that $\Pi(X) = Y$, corresponding to choices of sign for the $z_i$ in (\ref{eq:fanpallmatrix}). Furthermore, since the centralizer of $\l$ in $SO(n)$ consists of reflections through an even number of coordinate axes, for each $X$ there are $2^{n-1}$ choices of $k \in SO(n) - \mc{S}$ such that $k.\l = X$. 
    

    Let $\bold{E_{ij}}$ be the $n \times n$ matrix with a $1$ at position $(i,j)$, a $-1$ at $(j,i)$ and $0$ elsewhere. Then $\{ \bold{E_{ij}}k: 1 \leq i < j \leq n\}$ is an orthonormal basis for $T_k SO(n)$ (up to some constant). 
    Let $\bold{F_{ij}}$ be the $n-1 \times n-1$ matrix with a $1$ at $(i,j)$ and $(j,i)$ (or a single $1$ if $i=j$). Then $\{\bold{F_{ij}} : 1 \leq i \leq j \leq n-1\}$ is a orthonormal basis for $T_{F(k)} \mc{A}$.

    We calculate
    \begin{align*}
        D_k F(\bold{E_{ij}} k) 
        &= \lim_{t \to 0} \frac{1}{t} 
        \left[ \Pi((\exp t \bold{E_{ij}})k.\l) - \Pi(k.\l)
        \right] 
        \\
        &= 
        \Pi([\bold{E_{ij}}, k.\l]) 
        \\
        &= 
        \begin{cases}
            (\mu_j - \mu_i) \bold{F}_{ij} & j < n 
            \\ 
            \sum_{h = 1}^{n-1} z_h (1+\delta_{ih}) \bold{F}_{ih} & j = n 
        \end{cases}
        .
    \end{align*}
    The Jacobian determinant has one nonzero term (the one that associates $\bold{E_{ij}}$ to $\bold{F_{ij}}$ for $j < n$ and $\bold{F_{ii}}$ for $j = n$). Hence
    \begin{align*}
    \det D_k F = 
    \prod_{1 \leq i < j \leq n-1} (\mu_j-\mu_i) \prod_{1 \leq i \leq n-1} 2 z_i
    &= 
    \pm 
    2^{n-1}
    \prod_{\substack{1 \leq i \leq n \\ 1 \leq j \leq n - 1}} \a{\l_i - \mu_j}^{1/2}
    ,
    \end{align*}
    which is nonvanishing everywhere. Thus $F$ defines a smooth $2^{2n-2}$-fold covering, and pushing forward along $F$ we see that the left hand side of (\ref{eq:transfer}) is proportional to (\ref{eq:middle}).

    On the other hand, define 
    \begin{align*}
    &G : SO(n-1) \times \mc{M} \to \mc{A} \\ 
    &G(k',\mu) = k'.\mu
    .
    \end{align*} 
    It suffices to compute the differential at $k' = 1$. Take $\{ \bold{E_{ij}'} : 1 \leq i < j \leq n-1 \}$ as a basis for $T_{1}SO(n-1)$ and let $(e_i')_{1 \leq i \leq n-1}$ be the obvious basis for $T_\mu \mc{M}$. Then 
    \begin{align*}
    D_{(1,\mu)} G(\bold{E_{ij}'} ,0) 
    &= 
    [\bold{E_{ij}'}, \mu] = (\mu_j - \mu_i)\bold{F_{ij}'}
    \\
    D_{(1,\mu)} G(0, e_i') 
    &= \bold{F_{ii}'}
    \end{align*}
    and the Jacobian determinant is $\pm \prod_{1 \leq i < j \leq n-1} \a{\mu_i-\mu_j}$, so the right hand side of (\ref{eq:transfer}) pushes forward to (\ref{eq:middle}).

\end{proof}
Define
\[
J_n(\l) := \int_\mc{M}  
    A_{n-1}(\mu) 
    J(\mu) \, d\mu 
    .
\]
The next step is to show $I_n(\sqrt{n-1}\l) \ll J_n(\l) \ll I_n(\l/\sqrt{n})$. Then it will suffice to prove $J_n(\l) \asymp \norm{\l}^{-n+1} L_n(\l)$. Applying Lemma \ref{lem:transfer} to Equation (\ref{eq:indupper}) we get
\begin{align*}
    I_n(\l) 
    &\ll 
    \int_\mc{M}  
    \int_{SO(n-1)}
    B'(k'.\mu) 
    \, dk' 
    J(\mu) \, d\mu 
    \\
    &= 
    \int_\mc{M}  
    I_{n-1}(\mu) 
    J(\mu) \, d\mu 
    .
\end{align*}
By Equation (\ref{eq:nonzerotrace}) we may restrict to the region where $\a{\Tr \mu} < \sqrt{n-1}$, in which we have $I_{n-1}(\mu) = I_{n-1}(\mu_0/\sqrt{1-(\Tr \mu)^2/(n-1)})$, where $\mu_0$ is the tracefree part of $\mu$. By monotonicity, this is $\leq I_{n-1}(\mu_0)$, by induction $I_{n-1}(\mu_0) \asymp A_{n-1}(\mu_0)$ and since $\norm{\mu_0 - \mu}$ is bounded, $\answer_{n-1}(\mu_0) \asymp \answer_{n-1}(\mu)$. 
The upper bound now reads
\[
I_n(\l)
\ll 
\int_{\substack{\mu \in \mc{M} \\ \a{\Tr \mu} < \sqrt{n-1} }}  
A_{n-1}(\mu) J(\mu) d\mu. 
\]
Scaling $\lambda$ and $\mu$ by $\sqrt{n-1}$ this is equivalent to 
\begin{align*}
    I_n(\sqrt{n-1} \l) 
    &\ll 
\int_{\substack{\mu \in \mc{M} \\ \a{\Tr \mu} < 1}}  
A_{n-1}(\sqrt{n-1}\mu) J(\mu) d\mu
,
\end{align*}
and since $\answer_{n-1}(\sqrt{n-1}\mu) \asymp \answer_{n-1}(\mu)$ the RHS is $\asymp J_n(\l)$, so $I_n(\sqrt{n-1}\l) \ll J_n(\l)$.

On the other hand, for the lower bound we have 
\begin{align*}
    I_n(\l) 
    &\gg 
    \int_\mc{M}  
    \int_{SO(n-1)}
    B'(\sqrt{n} k'.\mu) 
    \, dk' 
    J(\mu) \, d\mu 
    \\
    &= 
    \int_\mc{M}  
    I_{n-1}(\sqrt{n}\mu) 
    J(\mu) \, d\mu 
    ,
\end{align*}
and we may restrict the integral to the region where $\Tr \mu < 1/\sqrt{n}$, in which 
\begin{align*}
I_{n-1}(\sqrt{n}\mu) 
&= I_{n-1}(\sqrt{n}\mu_0/\sqrt{1-(\Tr \sqrt{n}\mu)^2/(n-1)}) 
\\
&\geq 
I_{n-1}(\sqrt{n}\mu_0/\sqrt{(n-2)/(n-1)})
\\
&\asymp 
\answer_{n-1}(\mu_0)
\\
&\asymp 
\answer_{n-1}(\mu)
. 
\end{align*}
So
\begin{align*}
    I_n(\l) \gg 
    \int_{\substack{\mu \in \mc{M}
    \\
    \a{\Tr \mu} < 1/\sqrt{n}}}
    \answer_{n-1}(\mu) J(\mu) \, d\mu 
    ,
\end{align*}
and scaling $\l$ and $\mu$ by $1/\sqrt{n}$ we get  $I_n(\l/\sqrt{n}) \gg J_n(\l)$. 

Our strategy for showing $J_n(\l) \asymp \norm{\l}^{-n+1} L_n(\l)$ will be to simplify $J_n(\l)$ to a convolution. Define 
\begin{align*}
    \mc{H} 
    &:=
    {
    \textstyle
    \{ 
    \mu \in \mc{M} : \a{\sum_i \mu_i} < 1 
    \}
    }
    \\
    F_i(x) &:= 
    \1_{[\l_i,\l_{i+1}]}
    \a{\l_i - x}^{-1/2}
    \a{\l_{i+1}- x}^{-1/2}
    \\
    \xt{i}{j} 
    &:= 
    \a{\mu_i - \mu_j}
    \a{\l_i - \mu_j}^{-1/2}
    \a{\l_{j+1} - \mu_i}^{-1/2}
    .
\end{align*}
Then 
\begin{align}\label{eq:main}
    J_n(\l) 
    = 
    \int_\mc{H} 
    (1+\norm{\mu})^{-n+2}
    L_{n-1}(\mu) 
    \prod_{1 \leq i < j \leq n-1} 
    \xt{i}{j} 
    \prod_{1 \leq i \leq n-1}
    F_i(\mu_i) 
    \, d\mu 
    . 
\end{align}
Note that
\[
F_i(x) 
\asymp 
\bold{1}_{[\l_i,\l_{i+1}]}
\abs{\l_i - \l_{i+1}}^{-1/2}
\left(
\abs{\l_i - x}^{-1/2}
+
\abs{\l_{i+1} - x}^{-1/2}
\right)
,
\]
implying $\int F_i(x) dx \asymp 1$. Also, $\xt{i}{j} \leq 1$ on $\mc{H}$. If we can eliminate the factors $(1+\norm{\mu})^{-n+2}$, $L_{n-1}(\mu)$, and $\xt{i}{j}$ then the integral becomes simply a convolution evaluated at 0: 
\begin{align*}\label{eq:conv}
\int_\H \prod_{1 \leq i \leq n-1}
F_i(\mu_i) \, d\mu
&= 
\int 
\dots 
\int 
{ \textstyle
\1_{[-1,1]}\left(-\sum_i \mu_i\right)
}
\prod_{1 \leq i \leq n-1}
F_i(\mu_i) \, d\mu_{n-1} \dots d\mu_1 
\\
&=
\left( \1_{[-1,1]} * \left( \bigast_{1 \leq i \leq n-1} F_i \right) \right) (0)
.
\end{align*}
Let $d = \l_n - \l_1$ and let $d_i = \l_{i+1} - \l_i$ be the ``spectral gaps.'' 
Then Lemma \ref{lem:main} follows from the following cases: 

First, by Lemma \ref{lem:generallower}, we have unconditionally that $J_n(\l) \gg \norm{\l}^{-n+1}$.  
\begin{itemize}
    \item Let $I$ be the index of the largest gap. Suppose $d_I > (1-1/100n)d$. Then:
    \begin{itemize}
        \item If $(n,I) \neq (4,2)$, then $L_n(\l) \asymp 1$ and it remains to show $J_n(\l) \ll \norm{\l}^{-n+1}$. See Lemma \ref{lem:onegapA}. 
        \item If $(n,I) = (4,2)$, then the exceptional term term in $L_n(\l)$ is large. See Lemma \ref{lem:onegapB}.
    \end{itemize}
    \item Otherwise, let $I < J$ be the two largest gaps, breaking ties arbitrarily. Then $(1-1/100n)d \geq d_I, d_J \geq d/100n^2$.
    \begin{itemize}
        \item If $(I,J) \neq (1,n-1)$ then $L_n(\l) \asymp 1$ and it remains to show $J_n(\l) \ll \norm{\l}^{-n+1}$. See Lemma \ref{lem:twogapA}.
        \item If $(I,J) = (1,n-1)$ then the main term in $L_n(\l)$ is large. See Lemma \ref{lem:twogapB}.
    \end{itemize}
\end{itemize}

\section{Convolution Bounds}

\subsection{The general lower bound}\label{sec:generallower}
\begin{lem}\label{lem:generallower}
    For all $\l$ with $\l_1 < \l_2 < \dots < \l_n$, we have $J_n(\l) \gg \norm{\l}^{-n+1}$.
\end{lem}
\begin{proof}
Applying the inequalities $L_n(\mu) \gg 1$ and $ (1+\norm{\mu})^{-n+2} \gg \norm{\l}^{-n+2}$ to the definition of $J_n(\l)$ gives 
\[
J_n(\l) 
\gg 
\norm{\l}^{-n+2}
\int_\mc{H} 
    \prod_{1 \leq i < j \leq n-1} 
    \xt{i}{j} 
    \prod_{1 \leq i \leq n-1}
    F_i(\mu_i) 
    \, d\mu 
    . 
\]
We define a subset of $\mc{H}$ on which the $\xt{i}{j}$ are bounded away from 0. Let $\theta = -\l_1/d$. From $(n-1)\l_1 + \l_n \leq \l_1 + \l_2 + \dots + \l_n = 0$ we deduce $\l_1 \leq -d/n$ and similarly $\l_1 + (n-1) \l_n \geq \l_1 + \l_2 + \dots + \l_n = 0 \implies (n-1)d \geq -n \l_1$. 
Thus $1/n \leq \theta \leq 1 - 1/n$. Let $m_i = \theta \l_i + (1-\theta) \l_{i+1}$. Then $\sum_{1 \leq i \leq n-1} m_i = 0$.
Define intervals $M_i \subset [\l_i, \l_{i+1}]$ centered on the $m_i$, given by 
\[
M_i = 
[
m_i - d_i / 2n 
,
m_i + d_i / 2n 
]
\]
and restrict the integral to the set
\[
\mathcal{H}'
= 
\{
\mu : \mu_i \in M_i , 
{\textstyle 
\abs{\sum_i \mu_i}} < 1
\}
\subset \mc{H}
.
\]
Then for $\mu \in \mc{H}'$ we have
\[
\frac{\abs{\mu_i - \mu_j}}{\abs{\mu_i - \l_{j+1}}}
= 
\frac{\abs{\mu_i - \l_j} + \abs{\l_j - \mu_j}}{\abs{\mu_i - \l_j} + \abs{\l_j - \l_{j+1}}}
\geq 
\frac{\abs{\l_j - \mu_j}}{\abs{\l_j - \l_{j+1}}}
\geq \frac{1}{2n}
,
\]
and similarly $\abs{\mu_i-\mu_j}/\abs{\l_i - \mu_j} \geq 1/2n$,
so $\kappa_{ij} \geq 1/2n$. 
Also, for $\mu_i \in M_i$ we have 
\begin{align*}
    F_i(\mu_i) 
    \asymp 
    \a{\l_i - \l_{i+1}}^{-1/2}
    \left(
    \a{\l_i - \mu_i}^{-1/2}
    +
    \a{\l_{i+1} - \mu_i}^{-1/2}
    \right)
    \gg d_i^{-1}
    .
\end{align*}
Applying the lower bounds for $\xt{i}{j}$ and $F_i$ gives
\[
J_n(\l) 
\gg
\norm{\l}^{-n+2}
\int_{\mc{H}'} 
\prod_{1 \leq i \leq n-1}
d_i^{-1} \1_{M_i}(\mu_i)
d\mu_i 
,
\]
and this is equivalent to a convolution 
\begin{align*}
J_n(\l) 
&\gg 
\norm{\l}^{-n+2}
G(0)
\\
G
&= 
\1_{[-1,1]} * 
\left(
\bigast_{1\leq i \leq n-1} d_i^{-1} \1_{M_i}
\right)
.
\end{align*}
Since $\sum_i m_i = 0$ we can translate each factor to be centered at 0,
\begin{align*}
G
&=
\1_{[-1,1]} * 
\left(
\bigast_{1\leq i \leq n-1} d_i^{-1} \1_{[-d_i/2n,d_i/2n]}
\right)
.
\end{align*}
Then by the following Lemma, $G(x)$ is maximized at $x=0$. 
\begin{lem}\label{lem:conv}
    For a sequence of reals $a_i \geq 0$ let 
    \[ 
    f_n(x) = \bigast_{1 \leq i \leq n} \1_{[-a_i, a_i]}(x)
    \]
    then $f_n$ is maximized at 0 for all $n$. 
\end{lem}
\begin{proof}
    The $f_n$ are obviously even. We prove a stronger statement, that $f_n$ is nonincreasing on $[0, \infty)$. The case $n=1$ is trivial. For $n > 1$ and $x > 0$ we have
    \begin{align*}
    f_n(x) 
    &= 
    \int \1_{[-a_n, a_n]} (x-y) f_{n-1}(y) dy 
    \\ 
    &= 
    \int_{x - a_n}^{x + a_n} f_{n-1}(y) dy 
    \end{align*}
    implying, for $0 < x < x'$
    \begin{align*}
    f_n(x) - f_n(x')  
    &=
    \int_{x-a_n}^{x+a_n}
    f_{n-1}(y) dy 
    - 
    \int_{x'-a_n}^{x'+a_n}
    f_{n-1}(y) dy 
    \\
    &=
    \int_{x-a_n}^{x'-a_n}
    f_{n-1}(y) dy 
    - 
    \int_{x+a_n}^{x'+a_n}
    f_{n-1}(y) dy 
    \\
    &=
    \int_x^{x'} f_{n-1}(y - a_n) - f_{n-1}(y + a_n) dy 
    ,
    \end{align*}
    and by induction the integrand is $\geq 0$. 
\end{proof}

Finally, note that $\int G(x) dx = 2\prod_i \norm{d_i^{-1} \1_{M_i}}_1 \gg 1$ and $G(x)$ is supported in an interval of length $\ll \norm{\l}$. Thus, $G(0) \gg \norm{\l}^{-1}$ and we are done. 
\end{proof}
\subsection{Monotone rearrangement}
To upper bound various convolutions, we introduce the following convenient tool \cite{rearrangement}. For $f : \R \to [0,\infty)$ define the level set 
\[
\{f > t\} = \{ x \in \R : f(x) > t \},
\]
and the ``layer cake representation'' of $f$,
\[
f(x) = \int \1_{\{f > t \}}(x) \, dt .
\]
For a set $S \subset \R$ define the \textit{rearranged} set $S^* = [0, \mu(S))$. Finally define the \textit{monotone rearrangement} 
\[
f^*(x) = \int \1_{\{f > t\}^*}(x) \, dt .
\]



For example, the monotone rearrangement of $\1_{[a,b]}$ is $\1_{[0,b-a]}$, and the monotone rearrangement of
\begin{align*}
    f(x) &= \1_{[a,b]} \abs{a - x}^{-1/2} \abs{b - x}^{-1/2}
\end{align*}
is 
\begin{align*}
    f^*(x) &= 
    \1_{[0,b-a]} (x/2)^{-1/2} ((b-a) - x/2)^{-1/2}
    \\ 
    &\ll 
    (b-a)^{-1/2} x^{-1/2}
    .
\end{align*}
We note basic properties such as $f \leq g \implies f^* \leq g^*$ and $(af)^* = a f^*$. The real workhorse is the $n$-ary \textit{Hardy-Littlewood Rearrangement inequality}:
\begin{lem} For $f_1, \dots, f_n : \R \to [0,\infty)$ we have
    \[
\int \prod_{i=1}^n f_i(x) dx \leq \int \prod_{i=1}^n f_i^*(x) \, dx 
.
\]
\end{lem}
\begin{proof}
    \begin{align*}
    \int_\mathbb{R} f_1 (x) \dots f_n(x) \, dx 
    &= 
    \int_\mathbb{R}
    \int_0^\infty \dots \int_0^\infty 
    \1_{\{ f_1 > t_1\} }(x)
    \dots 
    \1_{\{ f_n > t_n\} }(x)
    \, dt_n
    \dots 
    \, dt_1
    \, dx 
    \\
    &= 
    \int_0^\infty \dots \int_0^\infty 
    \mu(
    \{ f_1 > t_1 \} 
    \cap \dots \cap 
    \{ f_n > t_n \} 
    )
    \, dt_n
    \dots 
    \, dt_1
    \\
    &\leq 
    \int_0^\infty \dots \int_0^\infty 
    \mu(
    \{ f^*_1 > t_1 \} 
    \cap \dots \cap 
    \{ f^*_n > t_n \} 
    )
    \, dt_n
    \dots 
    \, dt_1
    \\
    &=
    \int_\mathbb{R}
    f^*_1 (x) \dots f^*_n(x)
    \,
    dx
    .
\end{align*}
\end{proof}
For $n=2$ this implies $\norm{f * g}_\infty \leq \langle f^*, g^* \rangle$, where $f*g$ is the convolution of $f$ and $g$. We also note the following inequalities for $0 < a < b$:
\begin{align}\label{eq:logaverage}
\int_a^b x^{-1/2} \left(\log' \frac{T}{x}\right)^k \, dx 
&\ll_k 
b^{1/2} \left(\log' \frac{T}{b}\right)^k 
\\
\int_a^b x^{-1} \left(\log' \frac{T}{x}\right)^k \, dx
&\ll 
\log 
\f{b}{a} 
\left( 
\log' \f{T}{a}
\right)^k
,
\end{align}
the first of which can be verified by applying integration by parts $k$ times. 
\subsection{One Large Gap}
In this section we assume there exists $I$ such that $d_I > (1-1/(100n))d$.

Let $d' = d - d_I =  \a{\l_1 - \l_I} + \a{\l_{I+1} - \l_n} \leq d/100n$. We show that $n$ and $I$ determine the positions of $\l_i$ and $\mu_i$ up to an $O(d')$ error. Write $\pm C$ for an error of absolute value $\leq C$. Then $\l_i = \l_1 \pm d'$ for $i \leq I$ and $\l_i = \l_n \pm d'$ for $i > I$. Writing
\begin{align*}
0 &= \sum_i \l_i 
\\
&= I\l_1 + (n-I)\l_n \pm nd' 
\end{align*}
we get $\l_n = Id/n \pm d'$ and $\l_1 = (I-n)d/n \pm d'$. 
Then
\begin{align}\label{eq:onegaplam}
    \l_i 
    = 
    \begin{cases}
    \displaystyle{\frac{(I-n)d}{n} \pm 2d'} & i \leq I 
    \\
    \displaystyle{\frac{Id}{n} \pm 2d'} & i > I 
    \end{cases}
\end{align}
and (using $\a{\sum_i \mu_i} \leq 1$ to find $\mu_I$)
\begin{align}\label{eq:onegapmu}
    \mu_i 
    = 
    \begin{cases}
        \displaystyle{\frac{(I-n)d}{n} \pm 2d'} & i < I 
        \\
        \displaystyle{\frac{(I-2n)d}{n} \pm (2nd'+1)} & i = I
        \\
        \displaystyle{\frac{Id}{n} \pm 2d'} & i > I 
    \end{cases}
\end{align}

\subsubsection{The upper bound when $(n,I) \neq (4,2)$}

\begin{lem}\label{lem:onegapA}
    Assume there exists $I$ such that $d_I > (1-1/(100n))d$ and $(n,I) \neq (4,2)$. Then $L_n(\l) \asymp 1$ and $J_n(\l) \asymp \norm{\l}^{-n+1}$. 
\end{lem}
By Equation \ref{eq:onegaplam} we have $\a{\l_2} \gg d$ so 
\begin{align*}
    \log' \f{\norm{\l}}{1 + \a{\l_2} + \a{\l_{n-1}}} \ll 1 
    .
\end{align*}
Also, if $n = 4$ but $I \neq 2$ then at least one of $\a{\l_1 - \l_2}$ and $\a{\l_3 - \l_4}$ is $\gg d$ so 
\begin{align*}
    \log' \f{\norm{\l}}{1 + \a{\l_1 - \l_2} + \a{\l_3 - \l_4}} \ll 1 
    .
\end{align*}
Thus $L_n(\l) \asymp 1$, and by Lemma \ref{lem:generallower}, $J_n(\l) \gg \norm{\l}^{-n+1}$, so it remains to show $J_n(\l) \ll \norm{\l}^{-n+1}$. 

If $I \geq 2$ then $\a{\mu_1} \gg d$, otherwise, if $I = 1$ then $\a{\mu_{n-1}} \gg d$. Thus $(1+\norm{\mu})^{-n+2} \asymp \norm{\l}^{-n+2}$.

If $n = 4$ and $I \neq 2$ then $\a{\mu_2} \gg d$ and $L_{n-1}(\mu) \asymp 1$. 

If $n \geq 5$, then at least one of $\a{\mu_2}$ and $\a{\mu_{n-2}}$ must be $\gg d$, so 
\[
\log' \f{\norm{\mu}}{1 + \a{\mu_2} + \a{\mu_{n-2}}} \asymp 1 
.
\]
If $n = 5$ then the exceptional factor in $L_{n-1}(\mu)$ is also $\asymp 1$ since at least one of $\a{\mu_1 - \mu_2}$ and $\a{\mu_3 - \mu_4}$ must be $\gg d$. Thus $L_{n-1}(\mu) \asymp 1$ for $n \geq 5$.

Bounding $\xt{i}{j} \leq 1$ for all $i,j$ we get 
\[
J_n(\l) \ll \norm{\l}^{-n+2} \int_\H \prod_i F_i(\mu_i) \, d\mu 
.
\]
By Equation (\ref{eq:onegapmu}) we have $\mu_I = (I-2n)d/n \pm (2nd' + 1)$. Then (taking $d$ sufficiently large) $\a{\l_I - \mu_I}$ and $\a{\l_{I+1}-\mu_I}$ are $\gg d$ so $F_I(t) \asymp \norm{\l}^{-1}$.
Finally 
\begin{align*}
    J_n(\l) 
    &\ll 
    \norm{\l}^{-n+1} 
    \int_\mc{H} 
    \prod_{i \neq I} F_i(\mu_i) \, d\mu
    \\
    &=
    \norm{\l}^{-n+1} 
    \left(
    \1_{[-1,1]} * \left( \bigast_{i \neq I} F_i \right)
    \right)(0)
    \\
    &\ll
    \norm{\l}^{-n+1} \prod_{i \neq I} \norm{F_i}_1 
    \\
    &\ll 
     \norm{\l}^{-n+1} 
\end{align*}
and we are done. 

\subsubsection{The case $(n,I) = (4,2)$}
\begin{lem}\label{lem:onegapB}
    Assume there exists $I$ such that $d_I > (1-1/(100n))d$ and $(n,I) = (4,2)$. Then 
    \[
    L_n(\l) \asymp \log' \frac{\norm{\l}}{1 + \a{\l_1-\l_2} + \a{\l_3-\l_4}}
    \],
    and $J_n(\l) \asymp \norm{\l}^{-n_1} L_n(\l)$. 
\end{lem}
The proof proceeds similarly to Lemma \ref{lem:onegapA}, except that we cannot eliminate $L_{n-1}(\mu) = \log' \norm{\l}/(1+\a{\mu_2})$. We write instead
\[
J_n(\l) 
\ll 
\norm{\l}^{-n+2}
\int_{-\infty}^{\infty}
L(\mu_2)
G(\mu_2) \, d \mu_2
. 
\]
where
\begin{align*}
    L(t) &= \log' \f{\norm{\l}}{1 + \a{t}}
    \\
    G &= \1_{[-1,1]} * F_1 * F_3
.
\end{align*}
Using Young's inequality and the fact that $\norm{F_i}_1 \asymp 1$ we have
\begin{align*}
J_n(\l) 
&\ll 
\norm{\l}^{-n+1}
\norm{L* \1_{[-1,1]} * F_1 * F_3}_\infty 
\\ 
&\leq 
\norm{\l}^{-n+1}
\norm{L* F_1}_\infty \norm{\1_{[-1,1]}}_1 \norm{F_3}_1
\\ 
&\ll 
\norm{\l}^{-n+1}
\norm{L* F_1}_\infty 
.
\end{align*}
 Then using the Hardy-Littlewood rearrangement inequality and equation (\ref{eq:logaverage}):
\begin{align*}
    \norm{L* F_1}_\infty 
    &\leq \langle L^* , F^*_1 \rangle 
    \\
    &\ll 
    \int_0^{d_1} \log' \f{\norm{\l}}{1+\a{x}} 
    d_1^{-1/2} x^{-1/2}
    \, dx
    \\
    &\ll \log' \f{\norm{\l}}{1 + d_1}
    ,
\end{align*}
so $J_n(\l) \ll \norm{\l}^{-n+1} \f{\norm{\l}}{1 + d_1}$ and similarly $J_n(\l) \ll \norm{\l}^{-n+1}  \log' \f{\norm{\l}}{1 + d_3}$. Together these imply 
\[
J_n(\l) \ll \norm{\l}^{-n+1} \log' \f{\norm{\l}}{1 + \a{\l_1 - \l_2} + \a{\l_3 - \l_4}}
\]
as desired. 
For the lower bound, we restrict to $\mc{H}'$, then since $\xt{i}{j} \gg 1$ for all $i < j$, and $(1+\norm{\mu}^{-2} \gg \norm{\l}^{-2}$ we have
\begin{align*}
    J_n(\l) 
    \gg 
    \norm{\l}^{-2}
    \int_{\mc{H}'}
    \log' \f{\norm{\l}}{1 + \a{\mu_2}}
    \prod_{1 \leq i \leq 3} F_i (\mu_i) \, d\mu
    .
\end{align*}
By Equation (\ref{eq:onegapmu}) we have $\a{\mu_2} \leq 1 + 4d'$ for $\mu \in \mc{H}'$, so 
\begin{align*}
    J_n(\l) 
    \gg 
    \norm{\l}^{-2}
    \log' \f{\norm{\l}}{2 + 4d'}
    \int_{\mc{H}'}
    \prod_{1 \leq i \leq 3} F_i (\mu_i) \, d\mu 
\end{align*}
and using the same argument as in Lemma \ref{lem:generallower} the integral is $\gg \norm{\l}^{-1}$. Finally since $d' = \a{\l_1 - \l_2} + \a{\l_3 - \l_4}$ we have 
\begin{align*}
    J_n(\l) 
    \gg 
    \norm{\l}^{-3}
    \log' \f{\norm{\l}}{1 + \a{\l_1 - \l_2} + \a{\l_3 - \l_4}}
    .
\end{align*}
\subsection{Two Large Gaps}
In this section we assume no index $I$ satisfies $d_I > (1-1/100n)d$. This implies there exist indices $I < J$ such that $d_I, d_J \geq d/100n^2$.

\subsubsection{The case $(I,J) \neq (1,n-1)$}
\begin{lem}\label{lem:twogapA}
    Suppose there exist $I < J$ with $(I,J) \neq (1,n-1)$ such that $d_I, d_J \geq d/100n^2$. Then $L_n(\l) \asymp 1$ and $J_n(\l) \asymp \norm{\l}^{-n+1}$. 
\end{lem}
\begin{proof}
If $\a{\l_2} + \a{\l_{n-1}} \leq d/100n^2$, then $d_i < d/100n^2$ for $2 \leq i \leq n-2$, forcing $(I,J) = (1,n-1)$. Hence $\a{\l_2} + \a{\l_{n-1}} \gg \norm{\l}$ and the main term in $L_n(\l)$ is $\ll 1$. The exceptional term of $L_n(\l)$ is also $\ll 1$ since one of $\a{\l_1 - \l_2}, \a{\l_3 - \l_4}$ must be $d_I$ or $d_J$. It remains to prove $J_n(\l) \ll \norm{\l}^{-n+1}$. 

If $n = 3$ then $(I,J) = (1,2)$ is forced, so we must have $n \geq 4$. By symmetry (replacing $\l_1,\dots,\l_n$ with $-\l_n, \dots, -\l_1$) we may assume $J \leq n-2$. 

First we rewrite $L_{n-1}(\mu)$ to depend solely on $\mu_J$: 
\begin{lem}
Under the assumptions of Lemma \ref{lem:twogapA},
\[
L_{n-1}(\mu) 
\ll 
\mathcal{L}(\mu_J) := 
\left( 
\log'
\f{\norm{\l}}{
\a{\l_J - \mu_J}
}
\right) 
\left( 
\log'
\f{\norm{\l}}{
\a{\l_{J+1} - \mu_J}
}
\right)^{n-2}
\left( 
\log'
\f{\norm{\l}}{
\a{\mu_J}
}
\right)^{n-3}
.
\]
\end{lem}
\begin{proof}

If $J = n-2$ then bound
\[
\left(
\log' 
\f{\norm{\mu}}{
1 + \a{\mu_2} + \a{\mu_{n-2}}
}
\right)^{n-3}
\ll 
\left(
\log' 
\f{\norm{\l}}{
\a{\mu_J}
}
\right)^{n-3}
.
\]
Otherwise, suppose $J \leq n-3$. Since $1 \leq I < J$ we must have $n \geq 5$. If $(I,J) = (1,2)$ then $\a{\mu_2} + \a{\mu_{n-2}} \geq \a{\mu_2 - \mu_3} \geq \a{\mu_2 - \l_3} = \a{\l_{J+1} - \mu_J}$ and
\[
\left(
\log' 
\f{\norm{\mu}}{
1 + \a{\mu_2} + \a{\mu_{n-2}}
}
\right)^{n-3}
\ll 
\log' 
\left(
\f{\norm{\l}}{
\a{\l_{J+1} - \mu_J}
}
\right)^{n-3}
.
\]
Otherwise, if $(I,J) \neq (1,2)$ then $n \geq 6$ is forced, and $\a{\mu_2} + \a{\mu_{n-2}} \geq \a{\mu_2 - \mu_{n-2}} \geq \a{\l_3 - \l_{n-2}} \geq d_J \gg \norm{\l}$, so 
\[
\left(
\log' 
\f{\norm{\mu}}{
1 + \a{\mu_2} + \a{\mu_{n-2}}
}
\right)^{n-3}
\ll 1 
.
\]
As for the exceptional factor in $L_{n-1}(\mu)$ that appears when $n=5$, if $J=3$ we bound
\[
\log' 
\f{\norm{\mu}}{
1 + \a{\mu_1-\mu_2} + \a{\mu_3-\mu_4}
}
\ll 
\log' 
\f{\norm{\l}}{
\a{\l_{J+1} - \mu_J}
}
,
\]
and if $J=2$ we bound
\[
\ll 
\f{\norm{\l}}{
\a{\l_J - \mu_J}
}
.
\]
\end{proof}
Since  $\norm{\mu} \gg \a{\mu_1 - \mu_{n-1}} \gg d_J \gg \norm{\l}$, we have $(1+\norm{\mu})^{-n+2} \ll \norm{\l}^{-n+2}$. 

Since $\a{\l_{J+1} - \mu_I} \geq d_J \gg \norm{\lambda}$ we have 
\begin{align*}
\xt{I}{J} 
&= \a{\l_I - \mu_J}^{-1/2} \a{\l_{J+1} - \mu_I}^{-1/2} \a{\mu_I - \mu_J}
\\
&\ll \norm{\l}^{-1/2} \a{\mu_I - \mu_J}^{1/2}
,
\end{align*}
and since $\a{\l_J - \mu_{n-1}} \geq d_J \gg \norm{\lambda}$ we have
\begin{align*}
\xt{J}{n-1} 
&= 
\a{\l_J - \mu_{n-1}}^{-1/2} \a{\l_n - \mu_J}^{-1/2} \a{\mu_J - \mu_{n-1}}
\\ 
&\ll 
\norm{\l}^{-1/2} \a{\mu_J - \mu_{n-1}}^{1/2}
\\
&\ll
\norm{\l}^{-1/2} \a{\l_n - \mu_J}^{1/2}
.
\end{align*}
Applying all the above inequalities, and $\xt{i}{j} \leq 1$ for all other pairs $i,j$, to Equation (\ref{eq:main}) and converting to a convolution we have 
\[
J_n(\l) \ll 
\norm{\l}^{-n+1} 
\int
F(t) G(-t)
\, dt
\]
where 
\begin{align*}
     F(t) &= 
     \int_{\mu_I + \mu_J = t} 
     F_I(\mu_I) F_J(\mu_J) 
     \abs{\mu_I - \mu_J}^{1/2}
     \abs{\mu_J - \l_n}^{1/2}
     \mathcal{L}(\mu_J)
     \, 
     d\mu_J
     \\ 
     G(s) &= \left( \1_{[-1,1]} * \left( \bigast_{i \neq I, J} F_i \right) \right)(s)
     .
\end{align*}
 The estimation of $F(t)$ is contained in Lemma \ref{lem:2d}, with 
\begin{align*}
    &A = \l_I&
    &B = \l_{I+1}&
    &C = \l_J& 
    &D = \l_{J+1}&
    &E = \l_n&
    &T = d,&
    &k = 2n-4&
\end{align*}
giving 
\begin{align*}
    F(t) \ll 
    \left( \log' \f{\a{\l_J - \l_{I+1}}}{\a{\l_{I+1} + \l_J -t}} \right)^{k+1}  + 
    \left( \log' \f{\a{\l_n - \l_{J+1}}}{\a{\l_I + \l_{J+1}-t}} \right)^{k+1}  + 
    \\
    \left( \log' \f{d}{\a{\l_I + \l_J -t}} \right)^{k} + 
    \left( \log' \f{d}{\a{\l_{I+1} + \l_{J+1} -t}}\right)^{k} 
    .
\end{align*}
Let $d' = d - d_I - d_J$. Let $H_1(t)$ be the sum of the first two log terms and $H_2(t)$ be the sum of the second two log terms. We have $\a{\l_J - \l_{I+1}}, \a{\l_n - \l_{J+1}} < d'$ so $H^*_1(t) \ll (\log' \f{d'}{t})^{k+1}$. Let $d_K$ be the next largest gap after $d_I$ and $d_J$, we have $d_K \geq d'/(n-3)$. Then  
\begin{align*}
    \int H_1(t) G(-t)\, dt 
    &\ll \norm{ H_1 * F_K }_\infty \prod_{i \neq I, J, K} \norm{F_i}_1
    \\
    &\ll \langle H^*_1, F^*_K \rangle 
    \\ 
    &\ll  \int_0^{d_K} \left( \log' \f{d'}{x} \right)^{k+1} d_K^{-1/2} x^{-1/2} \, dx 
    \\
    &\ll \left(\log' \f{d'}{d_K} \right)^{k+1}
    \\
    &\ll 1 
    ,
\end{align*} 
as desired. To bound $\int H_2(t) G(-t) dt$ there are two cases. First suppose $d' \geq d/(100n)$. Then 
\begin{align*}
    \int H_2(t) G(-t) \, dt 
    &\ll \langle H^*_2 , F^*_K \rangle 
    \\
    &\ll \int_0^{d_K} \left(
    \log' \frac{d}{x} 
    \right)^{k} d^{-1/2}_K x^{-1/2} \, dx 
    \\
    &\ll \left( \log' \f{d}{d_K} \right)^{k}
    \\
    &\ll 1 
    ,
\end{align*}
since $d_K \gg d$. Otherwise $d' < d/(100n)$. 
If $G(-t) \neq 0$ then  
\begin{align*}
    t 
    &\in -\supp G 
    \\
    &\subset [-1,1] + \sum_{\substack{1 \leq i \leq n-1 \\ i \neq I,J}} [-\l_{i+1}, -\l_i]
    \\
    &= [-1 + \l_1 + \l_{I+1} + \l_{J+1}, \, \, 1 + \l_n + \l_I + \l_J].
\end{align*}
Recall that $\l_1 < -d/n$. Subtracting $\l_I + \l_J$ from the lower bound for $t$ we get 
\begin{align*}
    t - \l_I - \l_J &\geq -1 + \l_1 + d_I + d_J 
    \\
    &\geq -1 - d/n + d_I + d_J 
    \\
    &\geq -1 + (1 - 1/100n -1/n)d
\end{align*}
so for sufficiently large $\l$, $\a{t - \l_I - \l_J} \gg \norm{\l}$, and similarly $\a{t - \l_{I+1} + \l_{J+1}} \gg \norm{\l}$. Hence $H_2(t) \ll 1$ on the support of $G(-t)$ and $\int H_2(t)G(-t) \, dt \ll 1$ as desired. 
\end{proof}

\subsubsection{The case $(I,J) = (1,n-1)$} 
\begin{lem}\label{lem:twogapB}
    Let $(I,J) = (1,n-1)$ and suppose $d_I, d_J \geq d/100n^2$. Then 
    \[
    L_n(\l) \asymp \left(
    \frac{\norm{\l}}{
    1+\a{\l_2}+\a{\l_{n-1}}
    }
    \right)^{n-2}
    ,
    \]
    and $J_n(\l) \asymp \norm{\l}^{-n+1} L_n(\l)$.     
\end{lem}
\begin{proof}
Clearly the exceptional term in $L_4(\l)$ is trivial when $d_1, d_3 \gg \norm{\l}$. We split into three cases: the upper bound for $n=3$, the upper bound for $n \geq 4$, and the lower bound. 
\subsubsection{The upper bound when $n=3$}
Since $\a{\l_1-\l_2},\a{\l_2-\l_3} \gg \norm{\l}$ we have
\begin{align*}
(1+\norm{\mu})^{-1} \xt{1}{2}
&\ll
\a{\l_1-\mu_2}^{-1/2}
\a{\l_3-\mu_1}^{-1/2}
\\
&\ll
\a{\l_1-\l_2}^{-1/2}
\a{\l_3-\l_2}^{-1/2}
\\
&\ll \norm{\l}^{-1}
.
\end{align*}
Then 
\begin{align*}
J_n(\l) 
\ll 
\norm{\l}^{-1}
\int_{-1}^1 \int_{\substack{
\mu_1 + \mu_2 = t
\\ 
\l_1 \leq \mu_1 \leq \l_2 \leq \mu_2 \leq \l_3 
}}
\a{\l_1 - \mu_1}^{-1/2}
\a{\l_2 - \mu_1}^{-1/2}
\a{\l_2 - \mu_2}^{-1/2}
\a{\l_3 - \mu_2}^{-1/2}
\, d\mu_1 \, dt 
.
\end{align*}
By Lemma \ref{lem:2dsmall} the inner integral is 
\[
\ll \norm{\l}^{-1} 
\left( 
\log' \f{\norm{\l}}{\a{2\l_2 - t}}
+
\log' \f{\norm{\l}}{\a{\l_1 + \l_3 - t}}
\right)
= \norm{\l}^{-1} 
\left( 
\log' \f{\norm{\l}}{\a{2\l_2 - t}}
+
\log' \f{\norm{\l}}{\a{\l_2 + t}}
\right)
,
\]
and 
\begin{align*}
    J_n(\l) &\ll 
    \norm{\l}^{-2} \int_{-1}^1 
    \log' \f{\norm{\l}}{\a{\l_2 + t}} 
    +
   \log' \f{\norm{\l}}{\a{2\l_2 - t}} 
    \, dt 
    \\
    &\ll \norm{\l}^{-2} \log' \f{\norm{\l}}{1 + \a{\l_2}}
\end{align*}
as desired.

For $2 \leq i \leq n-2$ we have 
\begin{align*}
    \xt{1}{i}
    &= \a{\mu_1 - \mu_i}
    \a{\l_{i+1} - \mu_1}^{-1/2}
    \a{\l_1 - \mu_i}^{-1/2}
    \\
    &\ll \a{\mu_1 - \mu_i}
    \a{\l_{i+1} - \mu_1}^{-1/2}
    \norm{\l}^{-1/2}
    \\
    &\ll \a{\mu_1 - \mu_i}^{1/2}
    \norm{\l}^{-1/2}
\end{align*}
and similarly $\xt{i}{n} \ll \a{\mu_i - \mu_n}^{1/2} \norm{\l}^{-1/2}$ so by AM-GM, $\xt{1}{i}\xt{i}{n} \ll \a{\mu_1 - \mu_{n-1}} \norm{\l}^{-1}$. Also,
$\xt{1}{n} \ll \a{\mu_1 - \mu_{n-1}} \norm{\l}^{-1}$.
We can use these factors to cancel $(1+\norm{\mu})^{-n+2}$, and the extra factor in $L_{n-1}(\mu)$ if $n=5$:
\begin{align*}
    \left(
    \log'
    \f{\norm{\mu}}
    {1+ \a{\mu_1-\mu_2} + \a{\mu_3-\mu_4}}
    \right)
    &(1+\norm{\mu})^{-n+2}
    \xt{1}{n-1}
    \prod_{2 \leq i \leq n-2} \xt{1}{i} \xt{i}{n-1}
    \\
    &\ll
    \log'
    \f{\norm{\mu}}{
    \a{\mu_1-\mu_2}
    } 
    \norm{\mu}^{-1/2} 
    \a{\mu_1 - \mu_2}^{1/2} 
    \norm{\l}^{-n+2} 
    \\
    &\ll 
    \norm{\l}^{-n+2} 
    ,
\end{align*}
where in the last line we use the fact that $x \log(1/x) \leq 1$ for $x \leq 1$. Bound $\xt{i}{j} \leq 1$ for all remaining $i,j$. 

Bound 
\begin{align*}
    \left(
    \log' \f{\norm{\mu}}
    {1 + \a{\mu_2} + \a{\mu_{n-2}}}
    \right)^{n-3}  
    &\ll 
    \left(
    \log' \f{\norm{\mu}}
    {1+\a{\mu_{n-2}}}
    \right)^{n-3}  
    .
\end{align*}
Write
\begin{align*}
    J_n(\l) 
    \ll 
    \norm{\l}^{-n+2}
    \int  F(t) G(-t) \, dt 
\end{align*}
where 
\begin{align*}
    F(t) &= 
     \int_{\substack{\mu_1 + \mu_{n-1} = t
     \\ \l_1 \leq \mu_1 \leq \l_2 
     \\ \l_{n-1} \leq \mu_{n-1} \leq \l_n 
     }}
     F_1(\mu_1) F_{n-1}(\mu_{n-1}) 
     \, d\mu_1
     \\
     G &= \1_{[-1,1]} * F^L_{n-2} *  \left(
     \bigast_{i \neq 1,n-2, n-1} F_i 
     \right)
     \\
     F^L_{n-2}(x) &= F_{n-2}(x)     
     \left(
    \log' \f{\norm{\mu}}
    {1+\a{x}}
    \right)^{n-3} 
    .
\end{align*}
Applying Lemma \ref{lem:2dsmall} and using $d_1, d_{n-1} \gg \norm{\l}$ gives 
\begin{align*}
    F(t) \ll 
    \norm{\l}^{-1} 
    \left( 
    \log' \f{\norm{\l}}{
    \a{\l_1 + \l_n - t}
    }
    +
    \log' \f{\norm{\l}}{
    \a{\l_2 + \l_{n-1} - t}
    }
    \right)
    .
\end{align*}
By symmetry (replacing $\l_1, \dots, \l_n$ with $-\l_n, \dots, -\l_1$) we may assume $\a{\l_{n-1}} \geq \a{\l_2}$, which also implies $\l_{n-1} \geq 0$. 
Now we split into three cases:
\begin{itemize}
    \item If $\l_{n-1} \leq 100$ then bound 
    \begin{align*}
        F^L_{n-2}(x) &\ll (\log' \norm{\l})^{n-3} F_{n-2}(x) 
    \end{align*}
    and
    \begin{align*}
    \int  F(t) G(-t) \, dt 
    &\ll 
    (\log' \norm{\l})^{n-3}
    \norm{F * \1_{[-1,1]}}_\infty 
    \prod_{i \neq 1,n-1} \norm{F_i}_1 
    \\
    &\ll
    \norm{\l}^{-1}
    (\log' \norm{\l})^{n-3}
    \int_0^2 
    \log' \f{\norm{\l}}{x}
    \, dx 
    \\ 
    &\ll 
    \norm{\l}^{-1} 
    (\log' \norm{\l})^{n-2}
    \\
    &\asymp 
    \norm{\l}^{-1}
    \left( \log' \f{\norm{\l}}{1+\a{\l_2} + \a{\l_{n-1}}}
    \right)^{n-2}
    \end{align*}
    and we are done. 
    \item Suppose $\l_{n-1} > 100$ and $\a{\l_2 - \l_{n-1}} \leq \a{\l_{n-1}}/100$. Then $\l_2, \dots, \l_{n-1} > 0$ and $\l_2 > 10$. Since $F^L_{n-2}$ is supported on $[\l_{n-2}, \l_{n-1}]$ we can bound
    \begin{align*}
        F^L_{n-2}(x) 
        &\ll 
        \left(
        \log' \f{\norm{\l}}{\a{\l_{n-2}}}
        \right)^{n-3} F_{n-2}(x) 
        \\
        &\ll 
        \left(
        \log' \f{\norm{\l}}{1 + \a{\l_2} + \a{\l_{n-1}}}
        \right)^{n-3} F_{n-2}(x) 
        .
    \end{align*}
    For $t$ satisfying $G(-t) \neq 0$ we have 
    \begin{align*}
        t &\in -\supp G 
        \\
        &\subset [-1 - \sum_{2 \leq i \leq n-2} \l_{i+1}, 1 - \sum_{2 \leq i \leq n-2} \l_i]
        \\
        &= [-1 + \l_1 + \l_2 + \l_n, 1 + \l_1 + \l_{n-1} + \l_n]
    \end{align*}
    so $t - \l_1 - \l_n \geq \l_2 - 1 \gg \a{\l_2} + \a{\l_{n-1}}$. On the other hand, 
    \begin{align*}
        \l_2 + \l_{n-1} - t 
        &\geq 
        \l_2 + \l_{n-1} - (1 + \l_1 + \l_{n-1} + \l_n) 
        \\
        &\geq \l_2 - 1 - \l_1 - \l_n 
        \\
        &\geq \l_2 - 1 + (\l_2 + \dots + \l_{n-1}) 
        \\
        &\geq (n-1)\l_2 - 1
        \\
        &\gg 1 + \a{\l_2} + \a{\l_{n-1}}
        .
    \end{align*}  
    Thus 
    \begin{align*}
        F(t) 
        &\ll 
        \norm{\l}^{-1} 
        \log' \f{
        \norm{\l}
        }{
        1 + \a{\l_2} + \a{\l_{n-1}}
        }
        ,
    \end{align*}
    and the result follows. 
    \item 
    Finally suppose $\a{\l_{n-1}} > 100$ and $\a{\l_2 - \l_{n-1}} \geq \a{\l_{n-1}}/100$. Then $\a{\l_2 - \l_{n-1}} \asymp 1 + \a{\l_2} + \a{\l_{n-1}}$. If $\l_{n-2} < \l_{n-1} / 2$, then $d_{n-2} > \a{\l_2 - \l_{n-1}}/4 \gg \a{\l_2} + \a{\l_{n-2}}$ and by monotone rearrangement, 
    \begin{align*}
    \int  F(t) G(-t) \, dt 
    &\ll 
    \norm{\l}^{-1}
    \norm{F * F^L_{n-2}}_\infty 
    \prod_{i \neq 1,n-2,n-1} \norm{F_i}_1 
    \\
    &\ll 
    \norm{\l}^{-1}
    \int_0^{d_{n-2}}
    d_{n-2}^{-1/2} x^{-1/2}
    \left( \log' \f{\norm{\l}}{x} \right)^{n-2} 
    \, dx 
    \\
    &\ll 
    \norm{\l}^{-1}
    \left( \log' \f{\norm{\l}}{d_{n-2}} \right)^{n-2}
    \\ 
    &\ll 
    \norm{\l}^{-1}
    \left( \log' \f{\norm{\l}}{1 + \a{\l_2} + \a{\l_{n-2}}} \right)^{n-2}
    \end{align*}
    and we are done. Otherwise if $\l_{n-2} > \l_{n-1}/2$ then 
    \begin{align*}
    F^L_{n-2}(x) 
    &\ll 
    \left(
    \log' \f{\norm{\l}}{\a{\l_{n-2}}}
    \right)^{n-3} F_{n-2}(x) 
    \\
    &\ll 
    \left(
    \log' \f{\norm{\l}}{1 + \a{\l_2} + \a{\l_{n-1}}}
    \right)^{n-3} F_{n-2}(x) 
    .
    \end{align*}
    Choose $2 \leq K \leq n-2$ such that $d_K$ is the next largest after $d_J$, then 
    \begin{align*}
        d_K &> (d-d_I-d_J)/(n-3) 
        \\
        &= \a{\l_2 - \l_{n-1}}/(n-3) 
        \\
        &\gg 1 + \a{\l_2} + \a{\l_{n-2}}
    \end{align*} and
    \begin{align*}
    \int  F(t) G(-t) \, dt 
    &\ll 
    \norm{\l}^{-1}
    \left(
    \log' \f{\norm{\l}}{1 + \a{\l_2} + \a{\l_{n-2}}}
    \right)^{n-3} 
    \norm{F * K}_\infty 
    \prod_{i \neq 1,K,n-1} \norm{F_i}_1 
    \\
    &\ll 
    \norm{\l}^{-1}
    \left(
    \log' \f{\norm{\l}}{1 + \a{\l_2} + \a{\l_{n-2}}}
    \right)^{n-3} 
    \int_0^{d_K}
    d_K^{-1/2} x^{-1/2}
    \log' \f{\norm{\l}}{x} 
    \, dx 
    \\
    &\ll 
    \norm{\l}^{-1}
    \left(
    \log' \f{\norm{\l}}{1 + \a{\l_2} + \a{\l_{n-2}}}
    \right)^{n-3} 
    \left( \log' \f{\norm{\l}}{d_K} \right)
    \\ 
    &\ll 
    \norm{\l}^{-1}
    \left( \log' \f{\norm{\l}}{1 + \a{\l_2} + \a{\l_{n-2}}} \right)^{n-2}
    .
    \end{align*}
\end{itemize}
\subsubsection{Lower bound when $n \geq 3$ and $(I,J) = (1,n-1)$}
By assuming $L_n(\l)$ is larger than some constant, we may assume $1 + \a{\l_2} + \a{\l_{n-1}} < d/(100n)$. This further implies $\a{\l_1 + \l_n} = \a{\l_2 + \dots + \l_{n-1}} \leq d/100$, and since $d = \l_n - \l_1$, we have $\l_1 = (\l_1 + \l_n - d)/2 \leq -d/2 + d/200 \leq -d/3$ and $\l_n \geq d/3$.

Similarly to Section \ref{sec:generallower}, we restrict the integral to the region 
\begin{align*}
    \mc{H}'' = \{ \mu \in \mc{H} : 
    \mu_i \in M_i \text{ for } 2 \leq i \leq n-2
    , \, \,
    \mu_1 < -d/4
    , \, \, 
    \mu_{n-1} > d/4
    \}
    ,
\end{align*}
on which we have $\xt{i}{j} \geq 1/4$ for all $i<j$, and $F_i \gg d_i^{-1} \1_{M_i}$ for $2 \leq i \leq n-2$. 
Since $\l_2 < \mu_2 < \mu_{n-2} < \l_{n-1}$ we have
\begin{align*}
    L_{n-1}(\mu) 
    \gg 
    \left(
    \log' \frac{\norm{\l}}{1 + \abs{\l_2} + \abs{\l_{n-1}}} \right)^{n-3}
    .
\end{align*}
Bound $\a{\l_2 - \mu_1}^{-1/2} \a{\l_{n-1} - \mu_{n-1}}^{-1/2} \gg \norm{\l}^{-1}$ and $(1+\norm{\mu})^{-n+2} \gg \norm{\l}^{-n+2}$. 
Using all the above inequalities gives 
\begin{align*}
    J_n(\l) \gg 
    &\norm{\l}^{-n+1}
    \left(
    \log' \frac{\norm{\l}}{1 + \abs{\l_2} + \abs{\l_{n-1}}} \right)^{n-3}
    \times 
    \\
    &\int_{\mc{H}''} 
    \a{\l_1-\mu_1}^{-1/2}
    \a{\l_n-\mu_{n-1}}^{-1/2}
    \prod_{2 \leq i \leq n-2} F_i(\mu_i)
    J(\mu) 
    d\mu
    ,
\end{align*}
which we rewrite as a convolution
\begin{align*}
J_n(\l)
&\gg
\norm{\l}^{-n+1}
\int F(t) G(-t) \, dt 
\\
F(t) 
&= 
\int_{\substack{
\mu_1 + \mu_{n-1} = t
\\ 
\l_1 \leq \mu_1 \leq -d/4
\\ 
d/4 \leq \mu_{n-1} \leq \l_n
}}
\abs{\l_1 - \mu_1}^{-1/2}
\abs{\l_n - \mu_{n-1}}^{-1/2}
d\mu 
\\
G
&=
 \bold{1}_{[-1,1]}
 *
\left(
 \bigast_{2 \leq i \leq n - 2} d_i^{-1} \bold{1}_{M_i}
 \right)
 .
\end{align*}
Since $M_i \subset [\l_2, \l_{n-1}]$ for $2 \leq i \leq n-2$, we have $G(-t) \neq 0 \implies \a{t} \leq d/100$. If $t \leq \l_1 + \l_n$, then substituting $y = \mu_1 - \l_1$ in the equation for $F$ gives 
\begin{align*}
    F(t) &= 
    \int_0^{t - d/4 - \l_1} y^{-1/2} 
    (\l_n - (t - (y + \l_1)))^{-1/2} 
    \, dy 
    \\
    &=
    \int_0^{t- d/4 - \l_1} y^{-1/2} (y + \l_1 + \l_n - t)^{-1/2} \, dy 
    .
\end{align*}
 Apply Lemma \ref{lem:1dsmall} with $a = \l_1 + \l_n - t$ and $T = t - d/4-\l_1$. The preceding inequalities for $\a{t}, \l_1$, and $\a{\l_1+\l_n}$ ensure $a \leq T$. This gives 
\begin{align*}
    F(t) \gg \log' \f{\norm{\l}}{\l_1 + \l_n - t}
    .
\end{align*}
Likewise if $t \geq \l_1 + \l_n$ then 
\begin{align*}
    F(t) \gg \log' \f{\norm{\l}}{\l_1 + \l_n - t}
    ,
\end{align*}
thus 
\begin{align*}
F(t) 
\gg
\log' \frac{\norm{\l}}{\a{t - (\l_1 + \l_n)}}
. 
\end{align*}
Since $\a{\l_1 + \l_n} \ll \a{\l_2} + \a{\l_{n-1}}$ we have $\a{t - (\l_1 + \l_n)} \ll \a{t} + \a{\l_2} + \a{\l_{n-1}}$ and 
\begin{align*}
F(t) 
\gg
\log' \frac{\norm{\l}}{\a{t} + \a{\l_2} + \a{\l_{n-1}}}
. 
\end{align*} 
Meanwhile, we have $\int G(t) \, dt \gg 1$, and the support of $G$ is contained in an interval of radius $\ll 1 + \a{\l_2} + \a{\l_{n-1}}$ around 0, so 
\begin{align*}
    \int F(t) G(-t) \, dt 
    \gg 
    \log' \frac{\norm{\l}}{1 + \a{\l_1} + \a{\l_n}}
\end{align*}
as desired. 
\end{proof}
\section{Lemmas}
Here we have separated out the main calculation for the two large gaps case.

\begin{lem}\label{lem:2d}
Let $A \leq B \leq C \leq D \leq E$, $T$,  
and $(L_i)_{1 \leq i \leq k}$ be
given with $\a{A-E} \leq T$. 
Then for $A+C \leq t \leq B + D$, 
\begin{align*}
    (*) = 
    \int_{\substack{
    x+y = t
    \\ 
    A \leq x \leq B
    \\
    C \leq y \leq D
    }}
    \left(
    \a{x - A}
    \a{x - B}
    \a{y - C}
    \a{y - D}
    \right)^{-1/2}
    \a{x-y}^{1/2}
    \a{y-E}^{1/2}
    \prod_{1 \leq i \leq k} \log' \frac{T}{\a{L_i - y}}
    \, dx
\end{align*}
is bounded above by 
\begin{align*}    
    &\ll_k 
    T \a{A-B}^{-1/2} \a{C-D}^{-1/2}
    \times 
    \\ 
    &\left[ 
    \left( \log' \f{\a{B-C}}{\a{B+C-t}} \right)^{k+1}  + 
    \left( \log' \f{\a{D-E}}{\a{A+D-t}} \right)^{k+1}  + 
    \left( \log' \f{T}{\a{A+C-t}} \right)^{k} + 
    \left( \log' \f{T}{\a{B+D-t}}\right)^{k} 
    \right]
    .
\end{align*}
\end{lem}
\begin{proof}
    For $x \in [P,Q]$ we have 
    \begin{align}\label{eq:simple}
    \a{x-P}^{-1/2}\a{x-Q}^{-1/2}
    \ll 
    \abs{P-Q}^{-1/2}
    \left( 
    \a{x-P}^{-1/2} + \a{x-Q}^{-1/2}
    \right)
    .
    \end{align}
    Using this identity twice with $(P,Q) = (A,B)$ and $(P,Q) = (C,D)$ and substituting $y=t-x$ gives
    \begin{align*}
    (*)
    \ll 
    \a{A-B}^{-1/2} \a{C-D}^{-1/2}
    \sum_{\substack{X \in \{ A,B\}
    \\
    Y \in \{ C, D \} }}
    I(X,Y)
    \end{align*}
    where $I(X,Y) =$
    \begin{align*}
    \int \limits_{\max(A,t-D)}^{\min(B,t-C)}
    \a{x-X}^{-1/2}
    \a{x-t+Y}^{-1/2}
    \a{2x-t}^{1/2}
    \a{x-t+E}^{1/2}
    \prod_{1 \leq i \leq k} \log' \frac{T}{\a{L_i - t + x}}
    dx
    .
    \end{align*}

    First consider $I(A,C)$. Extend the bounds of the integral to $[A,t-C]$. Bound $\a{2x-t}^{1/2}\a{x-t+E}^{1/2} \ll T$.
    Then
    \[
    I(A,C) \ll 
    T \int_A^{t-C} 
    \a{x-A}^{-1/2} 
    \a{x-(t-C)}^{-1/2}
    \prod_{1 \leq i \leq k} \log' \frac{T}{\a{L_i - t + x}}
    \, dx
    \]
    and using Equation (\ref{eq:simple}) with $(P,Q) = (A, t-C)$ and then using monotone rearrangement on each factor we obtain 
    \[
    I(A,C)
    \ll 
    T \a{t-C-A}^{-1/2}
    \int_0^{t-C-A} x^{-1/2} \left( \log' \f{T}{x} \right)^k \, dx 
    \ll_k 
    T
    \left( 
    \log' 
    \f{T}{\abs{A+C-t}}
    \right)^k
    .
    \]
    The calculation for $I(B,D)$ is similar to $I(A,C)$. 
    
    Now consider $I(A,D)$. Bound $\abs{2x-t}^{1/2} \ll T^{1/2}$. Suppose $t \leq A + D$. Substitute $x'=x-A$. The integral now runs from $0$ to $B-A$, extend it to $[0,T]$. The integral now reads 
    \[
    T^{1/2}
    \int_0^T 
    \abs{x'}^{-1/2} \abs{x'+A+D-t}^{-1/2}
    \abs{x'+A+E-t}^{1/2}
    \prod_{1 \leq i \leq k} \log' \frac{T}{\a{L_i - t + x' + A}}
    \, dx
    .
    \]
    Applying Lemma \ref{lem:1d} with $a = A + D - t$ and $b = A + E - t$ gives 
    \begin{align*}
    I(A,D) 
    &\ll_k 
    T \left(\log' \f{A + E - t}{A + D - t} \right)^{k+1}
    \\
    &= 
    T\left(\log'
    \left(
    1 + \f{E - D}{A + D - t}
    \right)
    \right)^{k+1}
    \\
    &\ll 
    T \left(\log' \f{\a{D-E}}{\a{A + D - t}} \right)^{k+1}
    .
    \end{align*}
    If $t \geq A + D$, then substitute $x' = x - t + D$, extend the bounds of the integral to $[0,T]$, and apply Lemma \ref{lem:1d} with $a = t - A - D, b = E - D$. The calculation for $I(B,C)$ is similar. 
\end{proof}

\begin{lem}\label{lem:2dsmall}
    Let $A \leq B \leq C \leq D$ be given with $\a{A - D} \leq T$. Then for $A + B \leq t \leq C + D$, 
    \begin{align*}
    \int_{\substack{
    x+y = t
    \\ 
    A \leq x \leq B
    \\
    C \leq y \leq D
    }}
    \left(
    \a{x - A}
    \a{x - B}
    \a{y - C}
    \a{y - D}
    \right)^{-1/2}
    \, dx
\end{align*}
is bounded above by 
\begin{align*}
    \ll 
    \a{A-B}^{-1/2}
    \a{C-D}^{-1/2}
    \left(
    \log' \f{T}{\a{B+C-t}}
    +
    \log' \f{T}{\a{A+D-t}}
    \right)
    .
\end{align*}
\end{lem}
\begin{proof}
    Substitute $A,B,C+T,D+T,D+2T,3T,t+T$ for $A,B,C,D,E,T,t$ in Lemma \ref{lem:2d}. We obtain that
    \begin{align*}
    &\int_{\substack{
    x+y = t+T
    \\ 
    A \leq x \leq B
    \\
    C +T \leq y \leq D +T
    }}
    \left(
    \a{x - A}
    \a{x - B}
    \a{y - C - T}
    \a{y - D - T}
    \right)^{-1/2}
    \a{x-y}^{1/2}
    \a{y-E-2T}^{1/2}
    \, dx
    \end{align*}
    is $\ll$ than
    \begin{align*}
    T
    \a{A-B}^{-1/2}
    \a{C-D}^{-1/2}
    \left(
    \log' \f{T+\a{B-C}}{\a{B+C-t}} + 
    \log' \f{T+\a{D-E}}{\a{A+D-t}}
    \right)
    .
    \end{align*}
     Note that $\a{x-y},\a{y-E-2T}$ in the integrand and $ T + \a{B-C}, T+\a{D-E}$ on the right hand side are all $\asymp T$, so dividing both sides by $T$ gives the desired inequality. 
\end{proof}

\begin{lem}\label{lem:1d}
For $0 < a, b < T$ and all $L_1, \dots, L_k$ we have
\begin{align*}
\int_0^T &x^{-1/2} (x+a)^{-1/2} (x+b)^{1/2} \prod_{1 \leq i \leq k} \log' \frac{T}{\abs{L_i + x}} \, dx
\ll_k
T^{1/2} \left(\log' \f{b}{a}\right)^{k+1}
.
\end{align*}
\end{lem}
\begin{proof}
Since $x^{-1/2} (x+a)^{-1/2} (x+b)^{1/2}$ is decreasing on $[0,T]$, we may use monotone rearrangement and assume $L_i = 0$ for all $i$. 
If $0 < a < b < T$ then the integral is 
\begin{align*}
    &\ll 
    \left[
    a^{-1/2} b^{1/2} \int_0^a x^{-1/2} 
    + b^{1/2} \int_a^b x^{-1}
    + \int_b^T x^{-1/2}
    \right] 
    \left(\log' \frac{T}{x}\right)^k \, dx
    \\ 
    &\ll_k
    b^{1/2} \left(\log' \frac{T}{a}\right)^k
    + b^{1/2} \left( \log' \frac{b}{a}\right) \left(\log' \frac{T}{a}\right)^k 
    + T^{1/2}
    \\ 
    &\ll_k
    T^{1/2} \frac{b^{1/2}}{T^{1/2}}
    \left( \log' \f{T}{b} + \log' \f{b}{a} \right)^k + T^{1/2} \f{b^{1/2}}{T^{1/2}} \left(\log \f{b}{a}\right) 
    \left(\log' \f{T}{b} + \log' \f{b}{a}\right)^k
    + T^{1/2}
    \\ 
    &\ll_k
    T^{1/2} \left(\log' \f{b}{a}\right)^{k+1}
\end{align*}
where on the third line, 
we use the fact that $(b/T)^{1/2}(\log' (T/b))^k \ll_k 1$.
Otherwise if $0 < b < a < T$ then $x^{-1/2} (x+a)^{-1/2} (x+b)^{1/2} \leq x^{-1/2}$
 and the integral is 
 \begin{align*}
\leq \int_0^T x^{-1/2} \left(\log' \frac{T}{x}\right)^k \, dx
\ll_k 
T^{1/2}.
 \end{align*}
 \end{proof}
 \begin{lem}\label{lem:1dsmall}
 For $0 < a \leq T$,
 \begin{align*}
     \int_0^T x^{-1/2} (x+a)^{-1/2} \, dx 
     \asymp \log' (T/a) 
     .
    \end{align*}
 \end{lem}
 \begin{proof}
         \begin{align*}
    \int_0^T 
    x^{-1/2} 
    (x+a)^{-1/2} \, dx
    &= 
    \int_0^a 
    x^{-1/2} 
    (x+a)^{-1/2} \, dx
    +
    \int_a^T
    x^{-1/2} 
    (x+a)^{-1/2} \, dx
    \\
    &\asymp 
    a^{-1/2}
    \int_0^a 
    x^{-1/2} 
    \, dx
    +
    \int_a^T
    x^{-1} \, dx
    \\
    &\asymp
    1
    +
    \log(T/a)
    .
    \end{align*}
 \end{proof}

\bibliographystyle{plain} 
\bibliography{refs}

\end{document}